\documentclass[11pt]{article}
\usepackage{amsmath,amsfonts,amssymb,amsthm,fancyhdr,bm,mathtools,enumitem}
\usepackage[hidelinks]{hyperref}
\usepackage{fullpage}
\usepackage[dvipsnames]{xcolor}
\usepackage[capitalize]{cleveref}
\usepackage[color=Purple!50!white]{todonotes}
\usepackage[numbers,sort]{natbib}

\newtheorem{thm}{Theorem}[section]
\newtheorem*{thm*}{Theorem}
\newtheorem{lem}[thm]{Lemma}
\newtheorem{prop}[thm]{Proposition}
\newtheorem{cor}[thm]{Corollary}
\newtheorem{conj}[thm]{Conjecture}
\newtheorem{qu}[thm]{Question}
\theoremstyle{definition}
\newtheorem{Def}[thm]{Definition}

\crefname{equation}{equation}{equations}
\crefname{lem}{Lemma}{Lemmas}
\crefname{thm}{Theorem}{Theorems}

\newlist{lemenum}{enumerate}{1}
\setlist[lemenum]{label=(\alph*), ref=\thelem(\alph*)}
\crefalias{lemenumi}{lemma} 

\DeclareMathOperator\pr{Pr}
\DeclareMathOperator\ext{ext}
\DeclareMathOperator\ex{ex}

\newcommand\dd{\mathrm{d}}
\newcommand\up[1]{^{(#1)}}

\newcommand\ab[1]{\lvert#1\rvert}

\newcommand{\flo}[1]{\lfloor #1 \rfloor}
\newcommand{\cei}[1]{\lceil #1 \rceil}

\newcommand\E{\mathbb{E}}
\newcommand\N{\mathbb{N}}

\newcommand\rhat{\hat r}

\title{Three early problems on size Ramsey numbers}
\author{David Conlon\thanks{Department of Mathematics, California Institute of Technology, Pasadena, CA 91125, USA. Email: {\tt dconlon\allowbreak{}@caltech.edu}. Research supported by NSF Award DMS-2054452.} \and Jacob Fox\thanks{Department of Mathematics, Stanford University, Stanford, CA 94305, USA. Email: {\tt jacobfox@stanford.edu}. Research supported by a Packard Fellowship and by NSF Awards DMS-1855635 and DMS-2154169.} \and Yuval Wigderson\thanks{School of Mathematics, Tel Aviv University, Tel Aviv 69978, Israel. Email: {\tt yuvalwig@tauex.tau.ac.il}. Research supported by NSF GRFP Grant DGE-1656518.}}
\date{}

\begin{document}
\maketitle
\begin{abstract}
	The size Ramsey number of a graph $H$ is defined as the minimum number of edges in a graph $G$ such that there is a monochromatic copy of $H$ in every two-coloring of $E(G)$. The size Ramsey number was introduced by Erd\H os, Faudree, Rousseau, and Schelp in 1978 and they ended their foundational paper by asking whether one can determine up to a constant factor the size Ramsey numbers of three families of graphs: complete bipartite graphs, book graphs (obtained by adding many common neighbors to the vertices of a clique), and starburst graphs (obtained by adding many pendant edges to each vertex of a clique). In this paper, we completely resolve the latter two questions and make substantial progress on the first by determining the size Ramsey number of $K_{s,t}$ up to a constant factor for all $t =\Omega(s\log s)$.
\end{abstract}

\section{Introduction}
Given two graphs $G$ and $H$, we say that \emph{$G$ is Ramsey for $H$} if every two-coloring of the edges of $G$ contains a monochromatic copy of $H$. Graph Ramsey theory is mainly concerned with determining which graphs $G$ are Ramsey for a given $H$. In particular, of central concern is
the \emph{Ramsey number} $r(H)$ of $H$, defined as the minimum number of vertices in a graph $G$ which is Ramsey for $H$.

In this paper, we study the \emph{size Ramsey number} $\rhat(H)$, defined as the minimum number of edges in a graph $G$ which is Ramsey for $H$. The size Ramsey number was introduced by Erd\H os, Faudree, Rousseau, and Schelp \cite{MR479691} in 1978. They proved several 
bounds on size Ramsey numbers, noting, for example, the basic inequality
\[\rhat(H) \leq \binom{r(H)}2\]
and presenting a proof, due to Chv\'atal, that this bound is tight when $H$ is a complete graph.
They ended their paper with four questions, asking for the asymptotic order of $\rhat(H)$ as $H$ ranges over four specific families of graphs. In this paper, we fully resolve two of these questions and make substantial progress on a third. The fourth question, about the size Ramsey number of paths, was resolved by Beck \cite{MR693028}, who proved the surprising result that $\rhat(P_n) = \Theta(n)$ for the path $P_n$ with $n$ vertices. This breakthrough inspired many of the subsequent developments in the field, such as the classic papers\ \cite{MR905153,haxell1995induced,MR1767025,MR2775894,MR1083592}
and the more recent results in~\cite{berger2019size, clemens2019size, clemens2021grid, CNT21, CNT23, draganic2020size, draganic_kolutovi, han2019size, han2020multicolour, kamcev2021size, letzter2021size}.

The first question asked by Erd\H os, Faudree, Rousseau, and Schelp \cite{MR479691} was about $\rhat(K_{s,t})$ for $s \leq t$. They proved the bounds
\[
	\Omega(st 2^s) \leq \rhat(K_{s,t}) \leq O(s^2 t 2^s),
\]
with the lower bound only holding for $t=\Omega(s^2)$. However, in a later paper \cite{MR1212883}, Erd\H os and Rousseau proved the lower bound $\rhat(K_{s,t})=\Omega(st2^s)$ for all  $s \leq t$.\footnote{They only state their result for $s=t$, but the proof carries through for all $s \leq t$. We present their proof, in this greater generality, in \cref{sec:complete-bipartite}.} More recently, Pikhurko \cite{MR1972078} found an asymptotic formula for $\rhat(K_{s,t})$ for all fixed $s$ and $t \to \infty$, which, in particular, implies that $\rhat(K_{s,t})=\Theta(s^2 t 2^s)$ for  $t$ sufficiently large in terms of $s$. However, Pikhurko's technique provides no quantitative estimate on how large $t$ must be for such a bound to hold.

Our first main result is an improved lower bound on $\rhat(K_{s,t})$ which holds for all $s \leq t$.
\begin{thm}\label{thm:kst-lb}
	For all $s \leq t$,
	\[
		\rhat(K_{s,t}) = \Omega \left(s^{2- \frac st} t 2^s\right).
	\]
\end{thm}
In particular, if $t \geq (1+\delta)s$ for any fixed $\delta>0$, then we get a power saving over the earlier lower bound of $\Omega(st 2^s)$. Moreover, once $t = \Omega(s\log s)$, the bound is tight up to a constant factor.

\begin{cor}
	If $t = \Omega(s\log s)$, then
	\[
		\rhat(K_{s,t}) = \Theta(s^2 t 2^s).
	\]
\end{cor}

The second question raised by Erd\H os, Faudree, Rousseau, and Schelp \cite{MR479691} concerned book graphs. Given positive integers $k$ and $n$, the \emph{book graph} $B_n\up k$ consists of $n$ copies of $K_{k+1}$, glued along a common $K_k$. Equivalently, it can be described as the join of a $K_k$ and an independent set of order $n$. This $K_k$ is called the \emph{spine} and the vertices of the independent set are called \emph{pages}. Ramsey numbers of book graphs play a central role in Ramsey theory, because all known techniques for proving upper bounds on the diagonal Ramsey numbers $r(K_t)$ 
rely on induction schemes that repeatedly use bounds on $r(B_n\up k)$ for appropriately chosen $k<t$ and $n$. These Ramsey numbers have received considerable attention of late, beginning with work of the first author~\cite{MR4115773}, who asymptotically determined $r(B_n\up k)$ for $k$ fixed and $n \rightarrow \infty$, and continuing with work of the authors~\cite{CFW, CFW21} giving alternative proofs and exploring variations of the  basic question. Regarding size Ramsey numbers, Erd\H os, Faudree, Rousseau, and Schelp \cite{MR479691} proved that 
\[
	\Omega(k^2 n^2) \leq \rhat(B_n\up k) \leq O(16^k n^2)
\]
for $n$ sufficiently large in terms of $k$. 
Thus, while they were able to prove that the dependence on $n$ is quadratic, there was a massive gap between the lower and upper bounds for the dependence on $k$. Our second main result closes this gap, determining $\rhat(B_n\up k)$ up to a constant factor for $n$ sufficiently large in terms of $k$.

\begin{thm}\label{thm:books}
	For every fixed $k \geq 2$ and all sufficiently large $n$, 
	\[
		 \rhat(B_n\up k) = \Theta(k2^k n^2).
	\]
\end{thm}

The third question raised in \cite{MR479691} concerns graphs which we call {starburst graphs} (they appear to have not been previously named in the literature). For positive integers $k$ and $n$, the \emph{starburst graph} $S_n\up k$ is obtained from $K_k$ by adding $n$ pendant edges to every vertex of $K_k$; thus, it has $kn+k$ vertices. Erd\H os, Faudree, Rousseau, and Schelp \cite{MR479691} proved\footnote{They only included the proof for the weaker lower bound $n^2/2$, but claimed the bound shown. For completeness, we include a proof of the lower bound in \cref{sec:starbursts}.} that if $k$ is fixed and $n$ is sufficiently large, then
\[
	\Omega(k^3 n^2) \leq \rhat(S_n\up k) \leq O(k^4 n^2).
\]
Thus, in this case, there is only a $\Theta(k)$ gap between the upper and lower bounds. Our final main result shows that the lower bound is tight up to the constant factor.

\begin{thm}\label{thm:starbursts}
	For every fixed $k \geq 2$ and all sufficiently large $n$, 
	\[
		\rhat(S_n\up k) = \Theta(k^3 n^2).
	\]
\end{thm}

The proofs of our main theorems are all relatively short, but they employ a surprising array of different techniques. In \cref{thm:kst-lb}, the main new idea is to use a 
random coloring with a hypergeometric distribution between certain vertices and their higher degree neighbors, rather than the uniform distribution that is usually used in Ramsey-theoretic lower bound constructions. The lower bound in \cref{thm:books} uses a degree-based random coloring, where the probability an edge is red depends on the degrees of its endpoints, while the upper bound uses some of the regularity techniques that were recently developed for studying the ordinary 
Ramsey numbers of books \cite{MR4115773,CFW,CFW21}. Finally, \cref{thm:starbursts} is proved by examining the properties of an appropriate random graph. 
We discuss each of these techniques, both at a high level and in detail, in the relevant sections.

The rest of the paper is organized as follows. We prove \cref{thm:kst-lb} in \cref{sec:complete-bipartite}, \cref{thm:books} in \cref{sec:books}, and \cref{thm:starbursts} in \cref{sec:starbursts}, concluding with some further remarks and open problems. We use $\log$ throughout to denote the base $2$ logarithm and $\ln$ for the natural logarithm. For the sake of clarity of presentation, we systematically omit floor and ceiling signs whenever they are not crucial. For the same reason, we make no serious attempt to optimize any of the constants appearing in our results.

\section{Complete bipartite graphs}\label{sec:complete-bipartite}

In order to obtain some intuition for our proof of \cref{thm:kst-lb}, it is helpful to briefly review the proofs of the existing bounds on $\rhat(K_{s,t})$, namely,
\[
	\Omega(st 2^s) \leq \rhat(K_{s,t}) \leq O(s^2 t 2^s)
\]
for $s \leq t$. We first consider the upper bound, due to Erd\H os, Faudree, Rousseau, and Schelp \cite{MR479691}.

\begin{prop}[\cite{MR479691}]\label{prop:kst-ub}
	For all $s \leq t$, $\rhat(K_{s,t}) \leq 4e s^2 t 2^s$. 
\end{prop}

\begin{proof}
Let $G$ be a complete bipartite graph with one part $A$ of order $2s^2$ and the other part $B$ of order $2et2^s$. We claim that $G$ is Ramsey for $K_{s,t}$, which implies the desired result since $G$ has $4es^2 t2^s$ edges. Consider a red/blue coloring of $E(G)$. Call a vertex in $B$ \emph{red} if at least half its incident edges are red and \emph{blue} otherwise. Without loss of generality, we may assume that there are at least $et2^s$ red vertices in $B$. Each such vertex contributes at least $\binom{s^2}s$ red stars $K_{1,s}$ whose central vertex is in $B$. Since there are exactly $\binom{2s^2}s$ $s$-tuples in $A$, one of them must appear as the set of leaves of such a star at least
\[
	et2^s \frac{\binom{s^2}s}{\binom{2s^2}s} \geq et2^s \cdot 2^{-s}  \left(1- \frac{1}{s^2}\right)\dotsb \left(1- \frac{s-1}{s^2}\right) \geq et \exp \left(-2\sum_{i=1}^{s-1}\frac i{s^2}\right) \geq t
\]
times, yielding a red $K_{s,t}$, where we used the inequality $1-x \geq e^{-2x}$, valid for $x \in [0,\frac 12]$.
\end{proof}

By being a little more careful in the proof of \cref{prop:kst-ub}, one can improve the bound by a factor of $8+o(1)$ to $\rhat(K_{s,t}) \leq (e/2+o(1))s^2 t2^s$, where the $o(1)$ term tends to $0$ as $s \to \infty$. This can be done 
by picking $A$ of order $s^2/2$, $B$ of order $(e+o(1))t2^s$, using both red and blue stars out of vertices of $B$ and applying convexity, though we also need to be a little more careful with the inequalities. This stronger upper bound is also present in \cite{MR479691} and
it follows from \cite[Theorem 4.6]{MR1972078} that it is asymptotically tight, i.e., that
\[
	\rhat(K_{s,t}) = \left(\frac e2+o(1)\right)s^2 t 2^s
\]
as long as $t$ is sufficiently large in terms of $s$, where the $o(1)$ term tends to $0$ as $s \to \infty$.

The lower bound argument, essentially due to Erd\H os and Rousseau \cite{MR1212883}, is somewhat more involved. Since $\rhat(K_{s,t})\geq \rhat(K_{s-2,t})$, it suffices to prove the following statement.

\begin{prop}[\cite{MR1212883}]\label{prop:weak-kst-lb}
	For all $t \geq s+2$, $\rhat(K_{s,t}) \geq st 2^s/100$.
\end{prop}

\begin{proof}
Since the desired result is obvious for $s = 1$, we will assume throughout that $s \geq 2$. Let $G$ be a graph with $q$ edges. We first give an upper bound for the number of copies of $K_{s,t}$ in $G$. Suppose $G$ has $N$ vertices $v_1,\dots,v_N$ and let $d_i = \deg(v_i)$. Assume, without loss of generality, that $d_1 \geq d_2 \geq \dotsb \geq d_N$. We recall the simple fact that $d_i \leq 2q/i$ for all $i$, which follows by double-counting the number of edges incident to $\{v_1,\dots,v_i\}$: it is at least $id_i/2$, but at most $q$, from which the bound follows. Let $N_i$ denote the number of copies of $K_{s,t}$ in $G$ where $v_i$ is the last (i.e., maximum index) vertex in the side of order $s$. Then
\begin{equation}\label{eq:N_i-bound}
	N_i \leq \binom{i-1}{s-1} \binom{d_i}t \leq \left(\frac {ei}s\right)^s \left(\frac{2eq}{it}\right)^t \leq \frac{(2e^2 q)^t}{s^s t^t i^{t-s}}.
\end{equation}
Therefore, the total number of $K_{s,t}$ in $G$ is at most
\[
	\sum_{i=s}^N N_i \leq \frac{(2e^2 q)^t}{s^s t^t} \sum_{i=s}^\infty \frac{1}{i^{t-s}} \leq \frac{(2e^2 q)^t}{s^s t^t} \int_{s-1}^\infty \frac{1}{x^{t-s}}\,\dd x = \left(\frac{2e^2 q}{s t}\right)^t \frac{s-1}{t-s-1}\left(\frac{s}{s-1}\right)^{t-s} \leq \left(\frac{8e^2 q}{s t}\right)^t,
\]
where we may start the sum at $i=s$ as $i$ is the maximum index on an $s$-set of vertices and we use our assumptions that $t \geq s+2$ and $s \geq 2$ to evaluate the integral. In the final inequality, we use the bounds $\frac{s-1}{t-s-1} \leq s \leq 2^s \leq 2^t$ and $\frac{s}{s-1} \leq 2$, which hold since $t \geq s+2$ and $s \geq 2$.

Now, suppose we color the edges of $G$ uniformly at random. Each copy of $K_{s,t}$ is monochromatic with probability $2^{1-st}$, so the expected number of monochromatic $K_{s,t}$ is at most $2(8e^2q/(st2^s))^t$. If we plug in $q =  st 2^s/100$, we find that the expected number of monochromatic $K_{s,t}$ is less than $1$, which yields the desired result.
\end{proof}

Note that in the upper bound argument, the worst case happens when all vertices in $B$ are incident with the same number of red and blue edges. However, in the lower bound argument, we colored all the edges uniformly at random, meaning that, for any given $b \in B$, the number of edges incident to $b$ that are colored red is binomially distributed. Usually, the difference between a binomial distribution and a hypergeometric distribution (where we condition on having equally many red and blue edges) is fairly minor, but it can have a significant effect when coloring, for instance, a complete bipartite graph between vertex sets $A$ and $B$ where $A$ is small.
This suggests using a slightly different procedure for the lower bound, where we randomly color the edges from low-degree to high-degree vertices hypergeometrically rather than uniformly. This intuition is indeed at the heart of our proof of \cref{thm:kst-lb}, although the actual coloring we use is somewhat more complicated: roughly speaking, we first dyadically partition according to degree and then use independent hypergeometric variables for each dyadic interval.

In the course of the proof, we will need the following simple analytic lemma.\footnote{We are grateful to Mehtaab Sawhney for a suggestion that greatly simplified the proof of this lemma.}
\begin{lem}\label{lem:lagrange}
	Let $m$ and $x$ be positive integers and let $x_1,\dots,x_m$ be non-negative integers summing to $x$. If $x \geq 2m$,
	 then
	\[
		\sum_{k=1}^m 2^{-k} (x_k^2-x_k) \geq \frac{x^2}{2^{m+3}}.
	\]
\end{lem}
\begin{proof}
	Let $P \subseteq [m]$ denote the set of $k \in [m]$ with $x_k \geq 1$. Then we have that
	\begin{align*}
		2^{m+1}\sum_{k=1}^m 2^{-k} (x_k^2 - x_k) &= 2^{m+1} \sum_{k \in P} 2^{-k} (x_k^2-x_k)\\
		&\geq 2^{m+1} \sum_{k \in P} 2^{-k} (x_k-1)^2 &&[y^2-y \geq (y-1)^2\text{ for } y \geq 1]\\
		&\geq \left(\sum_{k \in P} 2^{-k}(x_k-1)^2\right) \left(\sum_{k \in P} 2^k\right) &&[\textstyle \sum_{k \in P}2^k \leq \sum_{k=1}^m 2^k \leq 2^{m+1}]\\
		&\geq \left(\sum_{k \in P}(x_k-1)\right)^2 &&[\text{Cauchy--Schwarz}]\\
		&\geq (x-m)^2 \geq \frac{x^2}{4}.
	\end{align*}
	Dividing by $2^{m+1}$ gives the desired result.
\end{proof}

We are now ready to proceed with the proof of \cref{thm:kst-lb}. Once again, since $\rhat(K_{s,t})\geq \rhat(K_{s-2,t})$, it suffices to prove the following statement.

\begin{thm}
	For all $t \geq s+2$, $\rhat(K_{s,t}) \geq s^{2- \frac st}t2^s/10000$.
\end{thm}
\begin{proof}
	The desired result already follows from \cref{prop:weak-kst-lb} for $s<100$, so we henceforth assume that $s \geq 100$. 

	Let $L = \log \frac{s^2 t}{50(t-s)}$. We note, for future convenience, that $2^L \leq s^3 \leq t^3$, since $\frac{t}{t-s}\leq s$ for $t \geq s+2$. Additionally, since $s \geq 100$, we have that $s \geq 6 \log s \geq 2L$. 

	Let $q = s^{2- \frac st}t 2^{s}/10000$ and let $G$ be a $q$-edge graph with vertices $v_1,\dots,v_N$ and degree sequence $d_1 \geq \dotsb \geq d_N$, where $d_i = \deg(v_i)$. By adding isolated vertices to $G$ if necessary, we may assume that $L<\log(N+2)$.

	For every $1 \leq \ell \leq L$, we define the interval of vertices $I_\ell = \{v_i:2^{\ell}-1 \leq i \leq 2^{\ell+1}-2\}$. Let $A = \bigcup_{\ell=1}^L I_\ell$ and $B = V(G) \setminus A$. We color all edges inside $B$ uniformly at random. 
	All remaining edges have at least one endpoint in $A$ and we color these hypergeometrically, as follows. For every vertex $v_j \in V(G)$ and every $\ell \leq \min\{L,\cei{\log (j+2)-1}\}$, we pick a set $R_{j,\ell} \subset I_\ell$ uniformly at random among all sets with exactly $\ab{I_\ell}/2$ elements, making these choices independently for all $j$ and $\ell$. Then, for every $v_i \in I_\ell$ with $i<j$ that is adjacent to $v_j$, we color the edge $v_i v_j$ red if $i \in R_{j,\ell}$ and blue otherwise.\footnote{Note that, as described, the neighborhood of $v_j$ in $I_\ell$ isn't exactly hypergeometrically distributed, since we first pick a set of ``possible'' red neighbors hypergeometrically, but then only color red those ``possible'' neighbors which are truly neighbors in $G$. Though it doesn't really matter, this choice makes the analysis slightly simpler.

	There is also a slight subtlety when $v_j \in A$, in that, if $v_j \in I_\ell$ for some interval $I_\ell$, we only wish to describe the color of the edges from $v_j$ to the previous vertices $v_i \in I_\ell$; this is why we add the condition $i<j$.}

	One key observation we use several times is that, for every set $F$ of edges in $G$, the probability they are all red is at most $2^{-\ab F}$. To see this, we first note that it is true if $F$ consists of edges whose left endpoint lies in some fixed interval $I_\ell$ and whose right endpoint is some fixed vertex $v_j$. Indeed, in this case, the result follows immediately from the negative correlation property of the hypergeometric distribution: the probability that such an $F$ is monochromatic red equals the probability that the set of left endpoints is a subset of $R_{j,\ell}$, which is at most $2^{-\ab F}$. We can partition $F$ into subsets of this form, plus some additional edges lying in $B$. The events that each of these subsets is monochromatic red are all independent, so the probability that $F$ is monochromatic red equals the product of the probabilities that each is monochromatic red, which is at most $2^{-\ab F}$.

	Now suppose that $S \subseteq A$ is a set of vertices with $\ab S = s$. Let $s_\ell = \ab{S \cap I_\ell}$, so that $s = s_1+\dotsb+s_L$. Let $v_i$ be the rightmost vertex in $S$ (i.e., $i$ is the maximal index appearing in $S$) and let $L^*=\flo{\log (i+1)}$. By the choice of $L^*$, as $v_i \in S \subseteq A$, we have $v_i \in I_{L^*}$. We then have that $s_\ell=0$ if $\ell>L^*$, since there can be no element in $S \cap I_\ell$ for $\ell >L^*$ by the maximality of $i$.

	For $1 \leq \ell \leq L^*$ and for a vertex $v_j$ with $j>i$, the probability that $v_j$ is monochromatic red to $S\cap I_\ell$ is exactly
	\begin{align*}
		\frac{\ab {I_\ell}/2}{\ab {I_\ell}} \frac{\ab {I_\ell}/2-1}{\ab {I_\ell}-1}\dotsb \frac{\ab {I_\ell}/2-(s_\ell-1)}{\ab {I_\ell}-(s_\ell-1)} &\leq 2^{-s_\ell} \left(1-\frac 1{\ab{I_\ell}}\right)\left(1-\frac 2 {\ab{I_\ell}}\right)\dotsb \left(1- \frac{s_\ell-1}{{\ab{I_\ell}}}\right) \\
		&\leq 2^{-s_\ell} \exp \left(- 2^{-\ell} \binom{s_\ell}2\right)
		\\
		&= 2^{-s_\ell} e^{-2^{-\ell-1}(s_\ell^2-s_\ell)}.
	\end{align*}
	Therefore, the probability that $v_j$ is monochromatic red to $S$ is at most
	\[
		2^{-s} \exp \left(-\frac 12 \sum_{\ell=1}^{L^*} 2^{-\ell} (s_\ell^2-s_\ell)\right).
	\]
	We now apply \cref{lem:lagrange} with $x=s$ and $m = L^*$, which we may do since $s \geq 2L\geq 2L^*$, to conclude that the probability $v_j$ is monochromatic red to $S$ is at most
	\[
		2^{-s} \exp \left(-\frac 12 \frac{s^2}{2^{L^*+3}}\right)\leq 2^{-s} \exp\left(-\frac{s^2}{32 i}\right).
	\]

	As in the proof of \cref{prop:weak-kst-lb}, we let $N_i$ denote the number of copies of $K_{s,t}$ whose last vertex from the side of order $s$ is $v_i$. Similarly, for $v_i \in A$ and $0 \leq x \leq t$, let $N_{i,x}$ be the number of copies of $K_{s,t}$ whose last vertex from the side of order $s$ is $i$ and where exactly $x$ vertices from the side of order $t$ have index greater than $i$. 
	Let $y = t-x$. Then
	\[
		N_{i,x} \leq \binom{i-1}{s-1} \binom{i-s}{y}\binom{d_i}x \leq \left(\frac{ei}{s}\right)^s \left(\frac{ei}{y}\right)^y \left(\frac{2eq}{ix}\right)^x \leq \frac{(2e^2)^t q^x}{s^s y^y x^x i^{x-y-s}},
	\]
	where we take $0^0=1$ so that the bound also holds when $x=0$ or $x=t$.
	For each of the $x$ vertices to the right of $v_i$, they have probability at most $2^{-s} e^{-s^2/32i}$ of being monochromatic red to all $s$ chosen vertices. Additionally, by the negative correlation property mentioned earlier, the remaining $y$ vertices each have probability at most $2^{-s}$ of being monochromatic red to all $s$ chosen vertices.

	Therefore, the probability that some $K_{s,t}$ counted by $N_{i,x}$ is monochromatic red is at most 
	\begin{align}\label{eq:Nix-probability}
		\frac{(2e^2)^tq^x\cdot 2^{-st} e^{-s^2 x/32i}}{s^s y^y x^x i^{x-y-s}}  &= \frac{(2e^2)^ti^s}{s^s 2^{st}} \exp\left(x\ln q-\left(y\ln y + x\ln x + (x-y)\ln i + \frac{s^2 x}{32i}\right) \right),
	\end{align}
	where we let $0\ln0=0$, agreeing with our earlier convention that $0^0=1$. We claim that this probability is maximized when $x=t$. To see this, let 
	\[
		f(x) = x\ln q - \left((t-x)\ln(t-x)+x \ln x + (2x-t)\ln i + \frac{s^2 x}{32i}\right).
	\]
	For $x \leq t-1$, we can compute that
	\[
		f'(x) = \ln q - \left(-\ln(t-x)+\ln x + 2\ln i + \frac{s^2}{32i}\right)\geq \ln q - \left(\ln\frac{xi^2}{t-x}+\frac{s}{32}\right) \geq \ln q - \ln\left(4ts^6 e^{s/32}\right),
	\]
	since we have $s\leq i \leq 2^{L+1} \leq 2s^3$ and $t-x \geq 1$. Note that for $s \geq 100$, we have that $2^s e^{-s/32} \geq (10s)^5$. Therefore, by our choice of
	$q$, we see that 
	\[
		\frac{q}{4ts^6 e^{s/32}} = \frac{s^{2- \frac st}t 2^{s}/10000}{4ts^6 e^{s/32}} > \frac{2^s e^{-s/32}}{(10s)^5} \geq 1
	\]
	 and thus $f'(x)>0$ for all $x \in [0,t-1]$.

	This shows that $f$ is maximized on this interval at $x=t-1$. Additionally, we have that
	\begin{align*}
		f(t)-f(t-1) &=\ln q- \left(t\ln t-(t-1)\ln(t-1)+2\ln i + \frac{s^2}{32i}\right)
		\geq \ln q - \ln \left(4et s^6 e^{s/32}\right)
		\geq 0,
	\end{align*}
	again by our choice of $q$, where we use the fact that
	\[
		t\ln t - (t-1)\ln (t-1) = \int_{t-1}^t (1+\ln x)\,\dd x \leq 1+\ln t =\ln(et).
	\]
	This implies that $f(x) \leq f(t)$ for all $0\leq x \leq t$, which shows that, as claimed, the probability in (\ref{eq:Nix-probability}) is maximized for $x=t$. Therefore, for $v_i \in A$, the probability that some $K_{s,t}$ counted by $N_{i}$ is monochromatic red is at most
	\[
		\sum_{x=0}^t \frac{(2e^2)^tq^x\cdot 2^{-st} e^{-s^2 x/32i}}{s^s y^y x^x i^{x-y-s}} \leq (t+1) \frac{(2e^2q)^t 2^{-st} e^{-s^2 t/32i}}{s^s t^t i^{t-s}} \leq \frac{(4e^2q)^t}{2^{st}s^st^t}\exp\left(- \left(\frac{s^2t}{32i} + (t-s)\ln i\right)\right).
	\]
	Let $g(i) = \frac{s^2t}{32i} +(t-s)\ln i$ and note that
	\[
		g'(i) = \frac{t-s}{i} - \frac{s^2 t}{32i^2}.
	\]
	This shows that $g(i)$ is minimized when $i = \frac{s^2t}{32(t-s)}$ and, therefore, the minimum value of $g$ is $(t-s)\ln \frac{es^2t}{32(t-s)}$. 
	Plugging this in, the probability that some $K_{s,t}$ counted by $N_i$ is monochromatic red is at most
	\[
		\frac{(4e^2q)^t}{2^{st}s^st^t} \left(\frac{32(t-s)}{es^2 t}\right)^{t-s} \leq \left(\frac{4e^2q}{st2^s}\right)^t \left(\frac{50(t-s)}{st}\right)^{t-s}.
	\]
	Summing this up, the probability that some $K_{s,t}$ whose side of order $s$ lies in $A$ is monochromatic red is at most
	\[
		\sum_{i=s}^{2^{L+1}}  \left(\frac{4e^2q}{st2^s}\right)^t \left(\frac{50(t-s)}{st}\right)^{t-s} \leq 2t^3 \left(\frac{4e^2q}{st2^s}\right)^t \left(\frac{50(t-s)}{st}\right)^{t-s}\leq \left(\frac{100q}{st2^s}\right)^t \left(\frac{50(t-s)}{st}\right)^{t-s},
	\]
	since $2^{L+1} \leq 2t^3 \leq 3^t$ for  $t \geq 6$.
	For all the remaining $K_{s,t}$, we use the negative correlation property to see that each one is monochromatic red with probability at most $2^{-st}$. Adding this up over all $i >2^L$ and using our bound $N_i\leq \frac{(2e^2 q)^t}{s^s t^t i^{t-s}}$ from \eqref{eq:N_i-bound}, we find that the probability any of these copies is monochromatic red is at most 
	\[
		\sum_{i=2^L+1}^N \frac{(2e^2q)^t}{s^s t^t i^{t-s}} 2^{-st} \leq \frac{(2e^2 q)^t}{s^s t^t 2^{st}}  \int_{2^L}^\infty \frac{1}{x^{t-s}}\,\dd x=\left(\frac{2e^2 q}{st 2^s}\right)^t  \frac{2^L}{t-s-1} \left(\frac{s}{2^L}\right)^{t-s}\leq \left(\frac{100 q}{st2^s}\right)^t\left(\frac{s}{2^L}\right)^{t-s},
	\]
	using our assumption $t\geq s+2$ to evaluate the integral and again using the bound $2^L \leq t^3 \leq 3^t$.
	In total, the probability that any $K_{s,t}$ is monochromatic red is at most
	\[
		\left(\frac{100q}{st2^s}\right)^t \left(\left(50\frac{t-s}{st}\right)^{t-s}+\left(\frac{s}{2^L}\right)^{t-s}\right).
	\]
	We now recall our choice of $2^L = \frac{s^2t}{50(t-s)}$, so that both terms in parentheses are equal. Therefore, in total, we find that the probability of a monochromatic red $K_{s,t}$ is at most 
	\[
		\left(\frac{100 q}{st2^s}\right)^t 2\left(50\frac{t-s}{st}\right)^{t-s} < \left(\frac{100q}{st2^s}\right)^t \left(\frac{100}{s}\right)^{t-s}=100^{-s} <\frac 12,
	\]
	by plugging in $q = s^{2-\frac{s}{t}}t2^s/10000$. By symmetry, the same estimate holds for the probability of a blue $K_{s,t}$, which yields the desired conclusion.
\end{proof}

\section{Book graphs}\label{sec:books}

We prove the lower and upper bounds in \cref{thm:books} separately. We begin with the lower bound, which is somewhat simpler.

\subsection{The lower bound}\label{subsec:books-lb}

In this section, we use the Chernoff bound in the following form.

\begin{lem}\label{lem:chernoff}
	Let $X_1,\dots,X_t$ be independent random variables taking values in $\{0,1\}$ and let $X = \sum X_i$. If $\E[X] \leq n/2$, then
	\[
		\pr(X \geq n) \leq e^{-n/6}.
	\]
\end{lem}

\begin{proof}
	It follows from Corollary 2.4 and Theorem 2.8 in \cite{MR1782847} that if $x \geq c\E[X]$ for some $c>1$, then $\pr(X \geq x) \leq e^{-c'x}$, where $c' =\ln c-1+\frac1c$. Plugging in $c=2$ (so that $x=n$), we find that $c' = \ln 2 -\frac 12>\frac 1{6}$, which yields the claimed bound.
\end{proof}

We also need the following simple analytic result, estimating a certain integral that arises in the proof.

\begin{lem}\label{lem:beta-integral}
	For every integer $k \geq 1$, 
	\[
		\int_0^1 \ln \left(1- \frac{y^{1/k}}{2}\right)\dd y \leq \frac 1k - \ln 2.
	\]
\end{lem}

\begin{proof}
	Denote the integral by $I$. We begin by changing variables to $u = y^{1/k}$, so that $\dd u = \frac 1k y^{(1-k)/k}\dd y = \frac 1k u^{1-k}\dd y$ or, equivalently, $\dd y = k u^{k-1}\dd u$. We get that
	\[
		I = \int_0^1 k u^{k-1} \ln \left(1- \frac u2\right)\dd u.
	\]
	We now integrate by parts, letting $\dd g = ku^{k-1}\dd u$ and $f = \ln(1- \frac u2)$. Then $g = u^k$ and $\dd f=\dd u/(u-2)$, so that
	\[
		I = \left[ u^k \ln \left(1- \frac u2\right)\right]_0^1 - \int_0^1 \frac{u^k}{u-2}\,\dd u = -\ln 2 + \int_0^1 \frac{u^k}{2-u}\,\dd u.
	\]
	This integral can be explicitly evaluated in terms of the beta function, but for our purposes it suffices to bound it: we note that $1/(2-u)\leq 1$ since $u \leq 1$ and, therefore,
	\[
		\int_0^1 \frac{u^k}{2-u}\,\dd u \leq \int_0^1 u^k \,\dd u=\frac{1}{k+1} \leq \frac 1k.
	\]
	Thus, $I \leq \frac 1k -\ln 2$, as claimed.
\end{proof}

We now prove our lower bound on $\rhat(B_n\up k)$. We only state the result for $k \geq 25$, which implies the lower bound in \cref{thm:books} by picking the implicit constant appropriately to deal with $2 \leq k \leq 24$.

\begin{thm}\label{thm:books-lb}
	For every $k \geq 25$ and every $n\geq 300k^2$, $\rhat(B_n\up k)\geq k 2^k n^2/1200$.
\end{thm}

\begin{proof}
	Let $G$ be a graph with at most $q = k 2^k n^2/1200$ edges, with vertices $v_1,\dots,v_N$. We order the vertices so that $\deg(v_1)\geq \dotsb \geq \deg(v_N)$. We may assume that $G$ has no isolated vertices, so that $N \leq 2q$.

	We fix $s =  k/3$. For $0 \leq i \leq s$, let $V_i = \{v_{in/10+1},\dots,v_{(i+1)n/10}\}$. Thus, $V_0$ consists of the $n/10$ vertices of $G$ of highest degree, $V_1$ consists of the next $n/10$ vertices of highest degree, and so on. For every $1\leq i\leq s$, every vertex in $V_i$ has degree at most
	\[
		\deg(v_{in/10}) \leq \frac{2q}{in/10} = \frac{20q}{in} = \frac{k2^k n}{60i} = \frac{s 2^k n }{20i} \eqqcolon D_i.
	\]
	Let $U = V(G) \setminus(V_0\cup \dotsb \cup V_s)$. On $V_0\cup \dots \cup V_s$, we use a Tur\'an coloring: we color all edges inside each $V_i$ red and all edges between $V_i$ and $V_j$ blue for all $0 \leq i <j \leq s$. All remaining edges are colored randomly, as follows. First, every edge inside $U$ is colored red or blue with probability $\frac 12$. Second, every edge between $V_i$ and $U$ is colored red with probability $p_i = \frac 12 (i/s)^{1/k}$ and blue with probability $1-p_i$. All these random choices are made independently. Note that since $p_0=0$, the edges between $V_0$ and $U$ are actually colored blue deterministically.

	If $Q$ is a copy of $K_k$ in $G$, let $\ext(Q)$ denote the number of extensions of $Q$ to a monochromatic $K_{k+1}$; equivalently, $\ext(Q)$ is the number of pages in the largest monochromatic book whose spine is $Q$. Note that $\ext(Q)=0$ if $Q$ is not itself monochromatic. We next argue that $\E[\ext(Q)] \leq n/2$ for every choice of $Q$. We then use the Chernoff bound and the union bound to conclude that with positive probability there is no monochromatic $B_n\up k$ (and, thus, that there is a coloring with no monochromatic $B_n\up k$).

	First, we deal with red books. Let $Q$ be a copy of $K_k$ in $G$, which we think of as a potential spine of a red book. Suppose first that $Q \cap V_0 \neq \varnothing$. Since $B_{n/2}\up k$ is connected and all edges between $V_0$ and $V(G) \setminus V_0$ are blue, any red copy of $B_{n/2}\up k$ with a vertex in $V_0$ must actually be entirely contained within $V_0$. But $\ab{V_0} = n/10<n/2$, so any red $B_{n/2}\up k$ cannot intersect $V_0$. 

	Next, suppose that $Q \cap V_i \neq \varnothing$ for some $i \in [s]$. Since all edges between $V_i$ and $V_j$ are blue for $i\neq j$, any red book whose spine is $Q$ must lie in $V_i \cup U$. Let $r = \ab{V_i \cap Q}$ be the number of vertices in the spine which lie in $V_i$. For a vertex $v \in V_i \setminus Q$ which is adjacent in $G$ to every vertex of $Q$, the probability that $v$ forms a red extension of $Q$ equals $p_i^{k-r}$. Similarly, for $u \in U \setminus Q$ adjacent to all vertices of $Q$, the probability that $u$ forms a red extension of $Q$ is $p_i^r \cdot (\frac 12)^{k-r} = 2^{-k}(2p_i)^r$. If $r < k$, then the number of choices for such a vertex $u \in U \setminus Q$ is at most $D_s$, since any such vertex must be a common neighbor of the non-empty set $U \cap Q$ and so, in particular, must lie in the neighborhood of some fixed vertex in $U$. Therefore, by the definitions of $p_i$ and $D_i$,
	\[
		\E[\ext(Q)] \leq p_i^{k-r} \ab{V_i} + 2^{-k} (2p_i)^r D_s \leq \frac n{10}+ \left(\frac is\right)^{r/k}\cdot\frac{n}{20} \leq \left(\frac1{10}+\frac{1}{20}\right)n \leq \frac n2.
	\]
If, instead, $r = k$, then there are most $D_i$ choices for a vertex which is joined to every vertex of $Q$, so that	
	\[
		\E[\ext(Q)] \leq \ab{V_i} + p_i^k D_i \leq \frac n{10}+ \frac is \cdot\frac{sn}{20i} = \left(\frac1{10}+\frac{1}{20}\right)n \leq \frac n2.
	\]

	Finally, we consider the case where $Q \subseteq U$. Since every vertex of $U$ has degree at most $D_s$, there are at most $D_s$ common neighbors of $Q$ in $U$, each of which forms a red extension of $Q$ with probability $2^{-k}$. Moreover, for all $i \in [s]$, each common neighbor of $Q$ in $V_i$ forms a red extension of $Q$ with probability $p_i^k$. Therefore,
	\[
		\E[\ext(Q)] \leq 2^{-k}D_s+ \sum_{i=1}^s p_i^k\ab{V_i} = \frac{n}{20} + \frac{n}{10}\sum_{i=1}^s 2^{-k} \frac is \leq \left(\frac{1}{20} + \frac{s}{10\cdot 2^k}\right)n \leq \frac n2,
	\]
	since $s \leq k \leq 2^k$.
	This shows that for every choice of $Q$ the expected number of red extensions of $Q$ is at most $n/2$.

	Next, we deal with blue books, so let $Q$ be a potential spine of a blue book. Since each $V_i$ is colored monochromatically in red, $Q$ contains at most one vertex from each $V_i$. Let $R = \{0 \leq i \leq s: V_i \cap Q \neq \varnothing\}$ be the set of $i$ such that $Q$ contains a vertex from $V_i$ and let $r=\ab R$. Then all the pages of a blue book whose spine is $Q$ must lie in $U \cup \bigcup_{j \notin R} V_j$. Since $r \leq s+1 <k$, at least one vertex of $Q$ must be in $U$, so the number of common neighbors of $Q$ in $U$ is at most $D_s$. Each of these common neighbors forms a blue extension of $Q$ with probability $2^{r-k}\prod_{i \in R}(1-p_i)$. Additionally, for $j \notin R$, the probability that a common neighbor of $Q$ in $V_j$ yields a blue extension of $Q$ equals $(1-p_j)^{k-r}$. Therefore,
	\begin{align}\label{eq:blue-exts}
		\E[\ext(Q)] &\leq D_s 2^{r-k} \prod_{i \in R} (1-p_i) + \sum_{j \notin R} (1-p_j)^{k-r} \ab{V_j} = \frac{n}{20} \prod_{i \in R} (2-2p_i) + \frac{n}{10} \sum_{j \notin R} (1-p_j)^{k-r}.
	\end{align}
	By definition, we see that $0 \leq p_i \leq \frac 12$ for all $i$. Moreover, for $j \geq 1$, we have that 
	\[
		p_j \geq \frac 12 \left(\frac 1s\right)^{1/k} \geq \frac 12 \left(\frac{1}{2^k}\right)^{1/k} = \frac 14,
	\]
	since $s \leq k \leq 2^k$. Additionally, $r \leq s+1 \leq k/2$, so $k-r \geq k/2$. Therefore,
	\[
		\sum_{j \notin R} (1-p_j)^{k-r} \leq \sum_{j \notin R} (1-p_j)^{k/2} \leq 1 + \sum_{j=1}^s (1-p_j)^{k/2} \leq 1 +\sum_{j=1}^s \left(\frac 34\right)^{k/2} \leq 1 + k\cdot \left(\frac 34\right)^{k/2} \leq 2,
	\]
	since $k(3/4)^{k/2} \leq 1$ for $k \geq 25$. This shows that the second term in (\ref{eq:blue-exts}) is at most $n/5$. For the first term, we recall that since $p_i \leq \frac 12$, every term in the product is at least $1$. Therefore, this product is maximized when $R = \{0,1,\dots,s\}$. In that case, we have
	\[
		\prod_{i=0}^s (2-2p_i) = 2^{s+1} \exp\left(\sum_{i=0}^s \ln(1-p_i)\right) = 2^{s+1}\exp \left(\sum_{i=0}^s \ln \left(1- \frac 12 \left(\frac is\right)^{1/k}\right)\right).
	\]
	To estimate this sum, we use the fact that the function $f(x) = \ln(1- \frac 12 (\frac xs)^{1/k})$ is decreasing to write
	\[
		\sum_{i=0}^s f(i)=\sum_{i=1}^s f(i) = \sum_{i=1}^{s}\int_{i-1}^{i} f(i)\,\dd x \leq \sum_{i=1}^s \int_{i-1}^i f(x)\,\dd x = \int_0^s f(x)\,\dd x.
	\]
	For this integral, we change variables to $y = x/s$ and $\dd y = \dd x/s$ to find that
	\[
		\int_0^s f(x)\,\dd x = \int_0^s \ln \left(1- \frac 12 \left(\frac xs\right)^{1/k}\right)\dd x = s\int_0^1 \ln \left(1- \frac{y^{1/k}}{2}\right)\dd y \leq s\left(\frac 1k-\ln 2\right),
	\]
	where the inequality is from \cref{lem:beta-integral}. Putting all this together, we see that
	\[
		\prod_{i \in R} (2-2p_i) \leq \prod_{i=0}^s (2-2p_i) \leq 2^{s+1} \exp\left(s\left(\frac 1k -\ln 2\right)\right) = 2 e^{1/3}.
	\]
	Thus, the first term in (\ref{eq:blue-exts}) is at most $\frac{e^{1/3}}{10}n<\frac n5$. 
	Plugging all this into (\ref{eq:blue-exts}), we conclude that for every choice of a potential blue spine $Q$,
	\[
		\E[\ext(Q)] \leq \frac{n}{5} + \frac n5 \leq \frac n2.
	\]

	Thus, for every $k$-tuple $Q$ of vertices, the expected number of pages in a monochromatic book whose spine is $Q$ is at most $n/2$. The random variable $\ext(Q)$ can be expressed as the sum of independent $\{0,1\}$-valued random variables. Therefore, \cref{lem:chernoff} implies that, for any fixed $Q$, we have $\pr(\ext(Q) \geq n) \leq e^{-n/6}$. Recalling that the number of vertices in $G$ is $N \leq 2q$, we see that the total number of choices for $Q$ is at most
	\[
		\binom Nk \leq N^k \leq (k2^k n^2)^k \leq \exp (k^2 +2k\ln n),
	\]
	since $k 2^k<e^k$ for $k\geq 25$. If $n \geq 300k^2$, then $\frac n6 - k^2 \geq \frac n8$ and $\ln n \leq \sqrt n$, so
	\[
		k^2 + 2k\ln n - \frac n6 \leq 2k\sqrt n - \frac n8 < 0,
	\]
	since $\sqrt n >16k$. Therefore, taking a union bound over all choices of $Q$, we see that, with positive probability, none of them will have at least $n$ monochromatic extensions, showing that there is a coloring of $G$ with no monochromatic $B_n\up k$.
\end{proof}

By being slightly more careful in the proof of \cref{thm:books-lb}, we can show that the bound already holds when $n\geq C k \log k$ for an appropriate constant $C$. To accomplish this, we count the expected number of $k$-cliques $Q$ in $G$ that extend to a monochromatic $B_n\up k$ as 
\[
	\sum_Q \pr(Q\text{ is monochromatic}) \pr(\ext(Q) \geq n \mid Q \text{ is monochromatic}).
\]
We then split this sum according to how many vertices of $Q$ are in $U$. Indeed, if we let $t$ be the number of vertices in $Q \cap U$, then Lov\'asz's version of the Kruskal--Katona theorem \cite[Exercise 13.31]{MR2321240} tells us that, provided $t \geq 2$, the number of ways of choosing the $t$ vertices of $Q$ in $U$ is at most $(2q)^{t/2}/t!$. By using this bound and carefully estimating the probability that $Q$ is monochromatic, one can obtain the claimed result.

However, such effort may be for naught, because this bound may not be optimal in this regime. Indeed, $\rhat(B_n\up k) \geq \rhat(K_k) = \binom{r(K_k)}2$. If the Ramsey number $r(k)$ is at least $(\sqrt 2+\varepsilon)^k$ for some fixed $\varepsilon>0$, as some have conjectured, then the lower bound in \cref{thm:books-lb} is not sharp for $n=2^{o(k)}$.

This lower bound argument points towards which host graph to use for the upper bound. Namely, it should have roughly $kn$ vertices of very high degree, each around $2^k n$, while all other vertices should have substantially lower degree. This suggests taking the host graph to be a large book, with a spine of order around $kn$ and about $2^k n$ pages, which is indeed what we do. However, the proof that such a large book is indeed Ramsey for $B_n\up k$ is quite involved, requiring substantial input from recent work on Ramsey numbers of books \cite{MR4115773,CFW,CFW21}. This proof is the content of the next subsection.

\subsection{The upper bound}
For the upper bound in \cref{thm:books}, we need to recall some facts related to Szemer\'edi's regularity lemma. Given vertex subsets $X,Y$ in a graph $G$, we denote by $e(X,Y)$ the number of pairs in $X\times Y$ that are edges of $G$ and define the \emph{edge density} by $d(X,Y) = e(X,Y)/(\ab X \ab Y)$. We write $d(X)$ if $X=Y$ and $d(x,Y)$ if $X=\{x\}$ is a singleton. A pair $(X,Y)$ is called \emph{$\varepsilon$-regular} if, for all $X' \subseteq X, Y' \subseteq Y$ with $\ab{X'} \geq \varepsilon \ab X, \ab{Y'}\geq \varepsilon \ab Y$, we have $\ab{d(X',Y')-d(X,Y)}\leq \varepsilon$. Similarly, we say that a set $X$ is $\varepsilon$-regular if the pair $(X,X)$ is $\varepsilon$-regular. A partition of the vertex set of a graph is called \emph{equitable} if all the parts have orders that differ by at most $1$. We use the following strengthened version of Szemer\'edi's regularity lemma, proved in \cite[Lemma 2.1]{CFW}.

\begin{lem}\label{lem:regularity}
	For every $\varepsilon>0$ and $M_0 \in \mathbb N$, there is some $M=M(\varepsilon,M_0) \geq M_0$ such that, for every graph $G$, there is an equitable partition $V(G)=V_1 \sqcup \dotsb \sqcup V_m$ into $M_0 \leq m \leq M$ parts such that the following hold: 
	\begin{enumerate}
		\item Each part $V_i$ is $\varepsilon$-regular and
		\item For every $1 \leq i \leq m$, there are at most $\varepsilon m$ values $1 \leq j \leq m$ such that the pair $(V_i,V_j)$ is not $\varepsilon$-regular.
	\end{enumerate}
\end{lem}

To complement the regularity lemma, we have the following consequence of the counting lemma, proved in \cite[Corollary 2.6]{CFW}. It is designed to count monochromatic extensions of cliques and thus estimate the size of monochromatic books. Given a copy $Q$ of $K_k$ in some graph $G$, we say that $u \in V(G)$ \emph{extends} $Q$ if $u$ is a common neighbor of all $k$ vertices of $Q$.

\begin{lem}\label{lem:randomclique}
	Fix $k \geq 2$ and let $\varepsilon,\delta \in (0,\frac 12)$ be parameters with $\varepsilon \leq \delta^{3k^2}$. 
	Let $X$ be a set of vertices in a graph $G$ and suppose that $X$ is $\varepsilon$-regular with edge density at least $\delta$. Then $X$ contains at least one $K_k$. Moreover, if $Q$ is a randomly chosen copy of $K_k$ in $X$, then, for any $u \in V(G)$,
	\begin{equation*}
		\pr(u \text{ extends }Q) \geq  d(u,X)^k- \delta.
	\end{equation*}
\end{lem}

The final tool we need concerns good configurations, which we now define. 

\begin{Def}
	Let $k \geq 2$ and $\varepsilon,\delta \in (0,1)$ be parameters and let $C_1,\dots,C_k$ be disjoint vertex subsets in a graph $G$ whose edges have been two-colored. We say that $C_1,\dots,C_k$ form a \emph{$(k,\varepsilon,\delta)$-good configuration} if the following conditions hold: 
	\begin{itemize}
		\item 
		$C_1\cup \dotsb \cup C_k$ induces a complete subgraph of $G$, meaning that there is an edge of $G$ between any two vertices of $C_1 \cup \dotsb \cup C_k$.
		\item  Each $C_i$ is $\varepsilon$-regular in red and has internal red density at least $\delta$. 
		\item Each pair $(C_i,C_j)$ for $i \neq j$ is $\varepsilon$-regular in blue and has blue density at least $\delta$.
	\end{itemize}
\end{Def}

The following lemma, which shows that colorings which contain good configurations also contain large monochromatic books,
was arguably the main idea in \cite{MR4115773}, but was first stated explicitly in~\cite{CFW}. The statement differs slightly from that in \cite[Lemma 3.3]{CFW}, but it is easy to check that the proof carries through.

\begin{lem}\label{lem:good-config-suffices}
	Fix $k \geq 2$, $0<\delta \leq 2^{-k-1}$, and $0<\varepsilon \leq \delta^{3k^2}$. Suppose that the edges of a graph $G$ have been two-colored and let $C_1,\dots,C_k$ be a $(k,\varepsilon,\delta)$-good configuration in $G$. If the vertices in $C_1\cup \dotsb \cup C_k$ have $t$ common neighbors in $G$, then the coloring contains a monochromatic $B_{2^{-k-1}t}\up k$.
\end{lem}

With these preliminaries, we are ready to prove the upper bound in \cref{thm:books}.

\begin{thm}
	For every $k \geq 2$ and every sufficiently large $n$, $\rhat(B_n\up k) \leq k2^{k+3}n^2$.
\end{thm}

\begin{proof}
	Let $G = B_N\up K$, where $K = 2kn$ and $N = 2^{k+1}n$, so that $G$ has $\binom K2+NK \leq 2k^2 n^2 + k2^{k+2}n^2 \leq k2^{k+3}n^2$ edges. We claim that $G$ is Ramsey for $B_n \up k$ if $n$ is sufficiently large with respect to $k$. To prove this, fix a red/blue coloring of $E(G)$. Let $\delta = \min \{2^{-k-1}, 1/(10k^3)\}$ and $\varepsilon = \delta^{3k^2}$.

	We apply \cref{lem:regularity} to the red graph on the subgraph of $G$ induced by the spine, with $\varepsilon$ as above and $M_0 = 10k^2$. We obtain an equitable partition $V_1 \sqcup \dotsb \sqcup V_m$ of the spine with $M_0 \leq m \leq M  = M(\varepsilon,M_0)$ parts, where each $V_i$ is $\varepsilon$-regular and, for every $i$, at most $\varepsilon m$ of the pairs $(V_i,V_j)$ are not $\varepsilon$-regular.
	Note that, since the colors are complementary, if a pair is $\varepsilon$-regular in red, then it is also $\varepsilon$-regular in blue.\footnote{This is true if the parts in the pair are disjoint. If not, one can choose a smaller $\varepsilon$ to guarantee the same end.} Call a part $V_i$ \emph{red} if at least half its internal edges are red and \emph{blue} otherwise. Without loss of generality, suppose that $m' \geq m/2$ of the parts are red and reindex so that these red parts are $V_1,\dots,V_{m'}$. We form a reduced graph $F$ with vertices $v_1,\dots,v_{m'}$, connecting $v_i$ to $v_j$ in $F$ if $(V_i,V_j)$ is $\varepsilon$-regular and has blue density at least $\delta$. 

	First suppose that some vertex in $F$, say $v_1$, has degree at most $(1- \frac{1}{k-1})m'$. This means that there are at least $\frac{m'}{k-1}-1$ parts $V_j$ with $2 \leq j \leq m'$ such that $(V_1,V_j)$ is either irregular or has blue density less than $\delta$. By the second condition in \cref{lem:regularity}, the number of irregular pairs $(V_1,V_j)$ is at most $\varepsilon m \leq 2 \varepsilon m'$. Therefore, if we let $J$ be the set of those $2 \leq j \leq m'$ such that $(V_1,V_j)$ is $\varepsilon$-regular and has blue density less than $\delta$, we find that
	\[
		\ab{J} \geq \frac{m'}{k-1} - 1 - 2\varepsilon m' = \left(\frac{1}{k-1}- \frac1 {m'} - 2 \varepsilon\right)m' \geq \left(\frac{1}{k-1} -\frac{2}{5k^2}\right)m' \geq \left(\frac{1}{k} + 3 k\delta\right)m',
	\]
	using the facts that $m' \geq m/2 \geq M_0/2 = 5k^2$ and $\varepsilon \leq \delta \leq 1/(10k^3)$, as well as $\frac{1}{k-1}- \frac{1}{k^2}\geq \frac 1k$.

	Let $U = \bigcup_{j \in J} V_j$. Since the partition is equitable, 
	\[
		\ab U \geq \frac{\ab J}{m}K-m \geq \left(\frac 1k +3k \delta\right)\frac K2 - M = \left(\frac 1k +3k \delta - 2 \frac MK\right) \frac K2 \geq \left(\frac 1k +2k \delta\right) \frac K2,
	\]
	where we used that $m \leq M$ and $2M/K \leq k \delta$ for $n$ (and thus $K$) sufficiently large.

	Recall that $V_1$ is $\varepsilon$-regular and has red density at least $\frac 12 \geq \delta$. Thus, by \cref{lem:randomclique}, if $Q$ is a random red $K_k$ in $V_1$, then, for any $u \in U$, the probability that $u$ extends $Q$ is at least $d_R(u,V_1)^k -\delta$, where $d_R$ denotes edge density in the red graph. Adding this up over all $u \in U$, we find that the expected number of red extensions of $Q$ is at least
	\[
		\sum_{u \in U} \left(d_R(u,V_1)^k - \delta\right) \geq \left(\frac{1}{\ab U}\sum_{u \in U} d_R(u,V_1)\right)^k \ab U - \delta \ab U \geq \left((1- \delta)^k - \delta\right) \ab U \geq (1-2k\delta)\ab U,
	\]
	where the first inequality follows from convexity of the function $x \mapsto x^k$ and the second uses the fact that $d_R(V_1,V_j)\geq 1- \delta$ for every $j \in J$. Plugging in our lower bound for $\ab U$, we find that the expected number of red extensions of $Q$ is at least
	\begin{align*}
		(1-2k \delta) \left(\frac 1k +2k \delta\right) \frac K2 \geq \left(\frac 1k +2k \delta - 2k \delta\right) \frac {2kn}2=n.
	\end{align*}
	Thus, there exists a red $K_k$ in $V_1$ with at least $n$ red extensions, that is, a red $B_n \up k$. 

	Thus, we may assume that every vertex in $F$ has degree greater than $(1- \frac{1}{k-1})m'$. By Tur\'an's theorem, this implies that $F$ contains a $K_k$. Let the vertices of this $K_k$ be $v_{i_1},\dots,v_{i_k}$ and let $C_1 = V_{i_1},\dots,C_k = V_{i_k}$. Then $C_1,\dots,C_k$ form a $(k,\varepsilon,\delta)$-good configuration, since each $C_i$ is $\varepsilon$-regular with red density at least $\frac 12 \geq \delta$ and all pairs $(C_i,C_j)$ with $i \neq j$ correspond to edges of $F$ and so are $\varepsilon$-regular with blue density at least $\delta$. Moreover, the vertices of $C_1 \cup \dotsb \cup C_k$ all lie in the spine of $G$ and, hence, they have $N = 2^{k+1}n$ common neighbors. Thus, by \cref{lem:good-config-suffices}, our coloring must contain a monochromatic $B_n \up k$.
\end{proof}

\section{Starburst graphs}\label{sec:starbursts}
\subsection{The lower bound}
Recall that the starburst graph $S_n\up k$ is obtained from $K_k$ by adding $n$ pendant edges to every vertex of $K_k$.
As mentioned in the introduction, although Erd\H os, Faudree, Rousseau, and Schelp~\cite{MR479691} asserted the lower bound $\rhat(S_n\up k) =\Omega(k^3 n^2)$ for sufficiently large $n$, they did not include a proof, so we begin with this lower bound. 
In fact, we record the following more general lower bound, where $\rhat(H_1,H_2)$ denotes the minimum number of edges in a graph $G$ such that every red/blue coloring of $E(G)$ contains a red copy of $H_1$ or a blue copy of $H_2$.

\begin{prop}\label{prop:general-starburst-lb}
	Let $H_1$ be a connected graph with $n+1$ vertices and maximum degree $\Delta$ and let $H_2$ be a graph with chromatic number $k+1$. Then $\rhat(H_1,H_2) \geq \binom k2 \Delta n$.
\end{prop}

Since the graph $S_n\up k$ has $kn+k$ vertices, maximum degree $n+k-1$, and chromatic number $k$, we immediately get the following corollary, which clearly implies the lower bound in Theorem~\ref{thm:starbursts} for $k \geq 3$ (see~\cite[Fact A]{MR479691} for the $k=2$ case).

\begin{cor}
	For every $k \geq 2$ and $n \geq 1$, 
	\[
		\rhat(S_n\up k)\geq \binom{k-1}2(n+k-1)(kn+k-1).
	\]
	In particular, $\rhat(S_n\up k) \geq (\frac 12-o(1))k^3 n^2$, where the $o(1)$ term tends to $0$ as $k \to \infty$.
\end{cor}

\begin{proof}[Proof of \cref{prop:general-starburst-lb}]
	Let $G$ be an $N$-vertex graph with fewer than $\binom k2 \Delta n$ edges. By iteratively deleting vertices of maximum degree, we may order the vertices of $G$ as $v_1,\dots,v_N$, where $v_i$ is a maximum-degree vertex in the induced subgraph $G_i = G[v_i,\dots,v_N]$ for every $i$. For $1 \leq i \leq N$, let $D(i)$ be the degree of $v_i$ in $G_i$ and declare $D(i)=0$ if $i>N$. By our choice of ordering, we have that $D(1)\geq D(2)\geq \dotsb$. 

	Suppose for the moment that $D(jn)\geq (k-j)\Delta$ for all $1 \leq j \leq k-1$. Then 
	the number of edges in $G$ is 
	\[
		\sum_{i=1}^N D(i) \geq \sum_{j=1}^{k-1} \sum_{i=(j-1)n+1}^{jn} D(i) \geq \sum_{j=1}^{k-1} n D(jn) \geq \Delta n\sum_{j=1}^{k-1}(k-j) = \binom k2 \Delta n,
	\]
	where the second inequality uses that $D(i) \geq D(jn)$ for all $i \leq jn$. This contradicts our assumption that the number of edges in $G$ is less than $\binom k2 \Delta n$. 
	So we may set $j^* \leq k-1$ to be the smallest positive integer such that $D(j^*n) < (k-j^*)\Delta$.
	We now create a partition of $V(G)$ into $V_1\sqcup \dotsb \sqcup V_k$, as follows.

	For $1 \leq j \leq j^*$, let $V_i = \{v_{(j-1)n+1},\dots,v_{jn}\}$ and let $G' = G \setminus (V_1 \cup \dotsb \cup V_{j^*})$. We consider a max $(k-j^*)$-cut of $G'$, that is, a partition of $V(G')$ into $k-j^*$ parts which maximizes the number of edges between the parts. Let this partition of $G'$ be $V_{j^*+1},\dots,V_k$. Note that $G' = G_{j^*n+1}$, which implies that every vertex in $G'$ has degree at most $D(j^*n) < (k-j^*)\Delta$. Moreover, in the max $(k-j^*)$-cut, every vertex has at least as many neighbors in every other part as in its own part, for otherwise we could increase the number of edges going between parts by moving some vertex to a different part. Thus, for every $j^*+1 \leq j \leq k$, every vertex in $V_j$ has fewer than $(k-j^*)\Delta/(k-j^*)=\Delta$ neighbors in its own part $V_j$.

	In summary, each part $V_j$ for $1 \leq j \leq j^*$ has exactly $n$ vertices and each part $V_j$ for $j^*+1 \leq j \leq k$ induces a subgraph of $G$ with maximum degree less than $\Delta$. We color the edges inside each part red and all edges between parts blue. Then the blue graph is $k$-partite and so cannot contain a copy of $H_2$. On the other hand, the red graph is the disjoint union of $G[V_1],\dots,G[V_k]$, so, since $H_1$ is connected, any red $H_1$ must appear in $G[V_j]$ for some $j$. But it cannot appear in $G[V_j]$ for $j \leq j^*$ since $H_1$ has more than $n$ vertices and it cannot appear in $G[V_j]$ for $j >j^*$ since $H_1$ has maximum degree $\Delta$ and $G[V_j]$ has maximum degree less than $\Delta$. Therefore, $G$ is not Ramsey for $H$.
\end{proof}

\subsection{The upper bound}

To prove the upper bound on $\rhat(S_n\up k)$, we need to recollect some facts about random graphs. We use $G(N,p)$ to denote the Erd\H os--R\'enyi random graph with $N$ vertices where each edge is chosen independently with probability $p$, saying that an event happens \emph{with high probability (w.h.p.)}\ if the probability it holds in $G(N,p)$ tends to $1$ as $N$ tends to $\infty$.

The first result we need is a version of Tur\'an's theorem relative to random graphs. This is a well-studied topic, culminating in works by Conlon--Gowers \cite{MR3548529} and Schacht \cite{MR3548528}, who determined the threshold for Tur\'an's theorem to hold in this setting (see also~\cite{MR3385638,MR3327533} for subsequent proofs using the method of hypergraph containers). However, we only need a rather weak version, so the following earlier result of Szab\'o and Vu~\cite{MR1999036}, with explicit bounds on the error probability, more than suffices for our purposes. For a graph $G$, we use $\ex(G,K_k)$ to denote the maximum number of edges in a $K_k$-free subgraph of $G$.

\begin{thm}[{\cite[Theorem 1.2]{MR1999036}}]\label{thm:random-turan}
	For every $k\geq 4$ and $\varepsilon>0$, there exist constants $C, c >0$ such that the following holds for all sufficiently large $N$. If $p \geq CN^{-2/(2k-3)}$ and $G \sim G(N,p)$, then
	\[
		\pr\left(\ex(G,K_k)\leq \left(1-\frac{1}{k-1}+\varepsilon\right)p\binom N2\right) \geq 1-2^{-c pN^2}.
	\]
\end{thm}

The following lemma records the facts about $G(N,p)$ used in our proof of the upper bound on $\rhat(S_n\up k)$. We use the notation $x = y\pm z$ to mean that $x$ lies in the interval $[y-z,y+z]$. Recall too that $d(X,Y)$ denotes the edge density between vertex sets $X,Y$.

\begin{lem}\label{lem:G(Np)-facts}
	Fix $k \in \N$ and $p,\delta \in (0, 1)$. The following facts all hold w.h.p.\ in $G \sim G(N,p)$:
	\begin{lemenum}
		\item\label{item:degrees} Every vertex $v \in V(G)$ has degree $(p\pm \delta)N$.
		\item\label{item:pair-density} For any two (not necessarily distinct or disjoint) sets $X,Y \subseteq V(G)$ with $\ab X, \ab Y \geq p^{8k}N$, $d(X,Y) = p \pm \delta$.
		\item\label{item:vtx-to-set} For every $S \subseteq V(G)$ with $\ab S \geq p^{8k}N$, the number of $v \in V(G)$ with fewer than $(p - \delta)\ab S$ neighbors in $S$ is at most $p^{8k}N$. Similarly, the number of $v \in V(G)$ with more than $(p+\delta)\ab S$ neighbors in $S$ is at most $p^{8k}N$.
		\item\label{item:turan} For every $S \subseteq V(G)$ with $\ab S \geq N/3$, every $K_k$-free subgraph of $G[S]$ has at most $(1- \frac{1}{2k})p \binom {\ab S}2$ edges.
	\end{lemenum}
\end{lem}

\begin{proof}
	It suffices to prove that each property holds w.h.p.,\ since a simple application of the union bound then implies that they all hold w.h.p.\ simultaneously. 
	\begin{enumerate}[label=(\alph*)]
		\item The expected degree of each vertex is $p(N-1)$, so, by the Chernoff bound (e.g., \cite[equation (2.12)]{MR1782847}), for any fixed $v \in V(G)$, 
		\[
			\pr(\deg(v) \neq (p\pm\delta)N) \leq \pr\left(\ab{\deg(v)-\E[\deg(v)]}\geq \frac \delta 2N\right) \leq 2\exp\left(-2 \left(\frac \delta 2\right)^2N\right),
		\]
		where we use that $N$ is sufficiently large in terms of $p$ and $\delta$ to write $p(N-1) = pN \pm \frac \delta 2 N$. We now take a union bound over all $N$ choices of $v$.
		\item \label{item:proof-of-pair-density} This also follows easily from an application of the Chernoff bound and the union bound. For details, see, e.g., \cite[Corollary 2.3]{MR2223394}.
		\item This follows immediately from part \ref{item:proof-of-pair-density}. Indeed, suppose we let $X$ consist of all vertices with fewer than $(p- \delta)\ab S$ neighbors in $S$. Then $d(X,S) < p- \delta$, which contradicts part (b) if $\ab X \geq p^{8k}N$. The second statement follows identically.
		\item For any fixed $S \subseteq V(G)$ with $\ab S \geq N/3$, we can apply \cref{thm:random-turan} with $\varepsilon = \frac{1}{2k}$ to conclude that the probability there is a $K_k$-free subgraph of $G[S]$ with more  than $(1- \frac{1}{2k})p \binom {\ab S}2$ edges is at most $2^{-c p(N/3)^2} = 2^{-\Omega_{p,k}(N^2)}$. Taking a union bound over the fewer than $2^N$ choices for $S$, we see that the required conclusion holds provided $N$ is sufficiently large with respect to $p$ and $k$.\qedhere
	\end{enumerate}
\end{proof}

The following result completes the proof of \cref{thm:starbursts}. Note that we can assume $k \geq 4$, since the required result for $k \in \{2,3\}$ follows from the $k = 4$ case.

\begin{thm}\label{thm:starbursts-ub}
	For every $k \geq 4$ and all sufficiently large $n$, $\rhat(S_n\up k) \leq 10^5 k^3 n^2$.
\end{thm}

\begin{proof}
	Fix $p = 1/(10k)$, $\delta=p^2$, and $N=1000k^2n$. Then, for $n$ sufficiently large, $G \sim G(N,p)$ satisfies all the properties in \cref{lem:G(Np)-facts} w.h.p.,\ so we may fix an $N$-vertex graph $G$ with these properties. In particular, by \cref{item:degrees}, $G$ has at most $pN^2 = 10^5 k^3 n^2$ edges, so it suffices to show that $G$ is Ramsey for $S_n\up k$. 

	Suppose, then, that the edges of $G$ have been red/blue colored. Let $V_R$ denote the set of vertices of $G$ whose blue degree is less than $3n$, let $V_B$ denote the set of vertices of $G$ whose red degree is less than $3n$, and let $V_0$ consist of all vertices with red and blue degree at least $3n$. At least one of $V_R, V_B,$ and $V_0$ has at least $N/3$ vertices.

	Suppose first that $\ab{V_R}\geq N/3$. By \cref{item:pair-density}, the number of edges in $V_R$ is at least $(p- \delta)\frac{\ab{V_R}^2}{2}$. Moreover, each vertex in $V_R$ has blue degree at most $3n$, which means that the number of blue edges in $V_R$ is at most
	\[
		\frac 32 n \ab{V_R} = \frac 32 \cdot \frac{N}{1000k^2}\ab{V_R} \leq \frac{9}{2000k^2}\ab{V_R}^2 \leq \delta \frac{\ab{V_R}^2}{2}.
	\]
	Therefore, the red graph on $V_R$ has at least $(p-2 \delta)\frac{\ab{V_R}^2}2\geq (1-2p)p\binom{\ab{V_R}}2$ edges. Since $2p = \frac{1}{5k}<\frac{1}{2k}$, \cref{item:turan} implies that $V_R$ contains a red $K_k$. Let the vertices of this red $K_k$ be $v_1,\dots,v_k$. Each $v_i$ has at least $(p -\delta)N$ incident edges, of which at most $3n$ are colored blue. Let $X_i$ be the set of red neighbors of $v_i$, apart from $v_j$ for $j \neq i$, so that
	\[
		\ab{X_i} \geq (p- \delta)N - 3n - k \geq \frac p2 N = 50kn.
	\]
	Now, we let $Y_1$ be an arbitrary subset of $X_1$ with $\ab{Y_1}=n$. Inductively, we let $Y_i$ be an arbitrary subset of $X_i \setminus(Y_1\cup \dotsb \cup Y_{i-1})$ of order $n$, which we can do since $\ab{X_i \setminus(Y_1 \cup \dotsb \cup Y_{i-1})}\geq 50kn - (i-1)n\geq n$. By doing this for every $i$, we find disjoint sets $Y_1,\dots,Y_k$ of red neighbors of $v_1,\dots,v_k$, respectively, which yields a red $S_n\up k$. By interchanging the roles of the colors, the same argument yields a blue $S_n\up k$ if $\ab{V_B}\geq N/3$.

	So we may now suppose that $\ab{V_0}\geq N/3$. Our goal is to apply the Erd\H os--Szekeres neighborhood-chasing argument to build a monochromatic $K_k$ inside $V_0$, while also maintaining large sets of red and blue neighbors for each vertex of this $K_k$, so that we can complete it to a monochromatic copy of $S_n\up k$. Formally, we inductively construct three sequences: a sequence $v_1,\dots,v_{2k}$ of vertices in $V_0$, a sequence $C_1,\dots,C_{2k} \in \{R,B\}$ of colors, and a sequence $Z_1,\dots,Z_{2k}$ of disjoint vertex sets, with the property that each $Z_i$ consists of $\frac 32 n$ red and $\frac 32 n$ blue neighbors of $v_i$. Additionally, we consider the sequence $N_1 \supseteq N_2 \supseteq \dots$, where $N_i = V_0\cap  N_{C_1}(v_1) \cap N_{C_2}(v_2)\cap \dotsb\cap N_{C_{i}}(v_i)$ consists of the common neighborhoods in $V_0$ of the $i$ vertices $v_1, \dots, v_i$ in the colors $C_1,\dots,C_i$. We inductively maintain the properties that $\ab{N_i}\geq p^{2i}\ab{V_0}$ and $v_{i+1} \in N_i$ for all $1 \leq i <2k$.

	To begin, by \cref{item:vtx-to-set}, at most $p^{8k}N<\ab{V_0}$ vertices in $V_0$ have fewer than $(p- \delta)\ab{V_0}$ neighbors in $V_0$, so let $v_1\in V_0$ be an arbitrary vertex with at least $(p- \delta)\ab{V_0}$ neighbors in $V_0$. Among the edges between $v_1$ and $V_0$, let $C_1$ be the majority color. Finally, since $v_1 \in V_0$, it has at least $3n$ red neighbors and at least $3n$ blue neighbors, so we let $Z_1$ be an arbitrary set of $\frac 32 n$ red and $\frac 32 n$ blue neighbors of $v_1$. As indicated above, let $N_1=V_0\cap N_{C_1}(v_1)$ be the set of neighbors of $v_1$ in $V_0$ in the color $C_1$. By construction,
	\[
		\ab{N_1}\geq \frac{p- \delta}{2} \ab{V_0} \geq \frac p4 \ab{V_0} \geq p^2 \ab{V_0},
	\]
	since $v_1$ has at least $(p-\delta)\ab{V_0}$ neighbors in $V_0$, of which at least half have color $C_1$. This sets up the base case of our induction.

	Inductively, suppose that for some $i<2k$ we have defined $v_1,\dots,v_i,C_1,\dots,C_i$, and $Z_1,\dots,Z_i$ and we wish to construct $v_{i+1},C_{i+1}$, and $Z_{i+1}$. As above, let $N_i = V_0 \cap N_{C_1}(v_1)\cap \dotsb \cap N_{C_i}(v_i)$, where the induction hypothesis and $i<2k$ implies that $\ab{N_i}\geq p^{2i}\ab{V_0} \geq p^{4k}N$. By \cref{item:vtx-to-set}, the number of vertices with fewer than $(p- \delta)\ab{N_i}$ neighbors in $N_i$ is at most $p^{8k}N$. Similarly, if we let $Z = Z_1\cup \dotsb \cup Z_i$, then $\ab Z = 3in\geq p^{8k}N$, so the number of vertices with more than $(p+\delta)\ab Z$ neighbors in $Z$ is at most $p^{8k}N$. Since $2p^{8k}N <p^{4k}N\leq \ab{N_i}$, we may pick some $v_{i+1}$ in $N_i$ with at least $(p- \delta)\ab{N_i}$ neighbors in $N_i$ and at most $(p+\delta)\ab Z$ neighbors in $Z$. Let $C_{i+1}$ be the majority color among all edges between $v_{i+1}$ and $N_i$, so that
	\[
		\ab{N_{i+1}} = \ab{N_i \cap N_{C_{i+1}}(v_{i+1})} \geq \frac{p- \delta}{2}\ab{N_i} \geq p^2 \ab{N_i} \geq p^{2(i+1)}\ab{V_0},
	\]
	maintaining the inductive hypothesis. Finally, to define $Z_{i+1}$, we recall that the number of neighbors of $v_{i+1}$ in $Z$ is at most
	\[
		(p+\delta)\ab Z \leq 2p\ab Z = \frac{6in}{10k}< \frac{12kn}{10k}<\frac 32 n.
	\]
	Since $v_{i+1} \in V_0$, it has at least $3n$ blue neighbors and at least $3n$ red neighbors. Hence, avoiding the neighbors of $v_{i+1}$ that are in $Z$, we can still find a set $Z_{i+1}$ of $\frac 32 n$ red and $\frac 32 n$ blue neighbors of $v_{i+1}$ which is disjoint from $Z_1,\dots,Z_i$.

	To conclude, we assume, without loss of generality, that at least half of $C_1,\dots,C_{2k}$ are red. That is, there are distinct $j_1,\dots,j_k \in [2k]$ such that $C_{j_1},\dots,C_{j_k}$ are all red. Then $v_{j_1},\dots,v_{j_k}$ form a red $K_k$, since, for any $j > i$, $v_j \in N_{j-1} \subseteq N_{C_i}(v_i)$. Moreover, the sets $Z_{j_1},\dots,Z_{j_k}$ are disjoint and each contains $\frac 32 n$ red neighbors of $v_{j_1},\dots,v_{j_k}$, respectively. By deleting at most $k<\frac n2$ vertices from each $Z_{j_i}$, we can ensure that each $Z_{j_i}$ is also disjoint from $\{v_{j_1},\dots,v_{j_k}\}$, while still containing at least $n$ red neighbors of $v_{j_i}$ for each $i$. This yields a red copy of $S_n\up k$, as desired.
\end{proof}

Observe that \cref{thm:starbursts}, which states that $\rhat(S_n\up k) = \Theta(k^3 n^2)$ for $n$ sufficiently large, actually requires that $n$ is at least exponential in $k$. This is because
\[
	\rhat(S_n\up k) \geq \rhat(K_k) = \binom{r(K_k)}2 \geq 2^k.
\]
A careful analysis of the proof of \cref{thm:starbursts-ub} shows that the upper bound $\rhat(S_n\up k) = O(k^3 n^2)$ holds when $n \geq k^{Ck}$ for an appropriate constant $C$. This particular proof, using an Erd\H os--R\'enyi random graph, cannot give a better bound on $n$. However, a different random graph model can be used to show that $\rhat(S_n\up k)=O(k^3 n^2)$ holds already when $n \geq 2^{Ck}$. Indeed, consider the random graph model $G(N,p,\omega)$ obtained by iteratively picking a uniformly random subset of $[N]$ of order $\omega$, adding a complete subgraph on those vertices, and repeating this process until the edge density is at least $p$. The global structure of $G(N,p,\omega)$ is rather close to that of $G(N,p)$ for an appropriate range of parameters, but the two models differ because $G(N,p,\omega)$ has larger cliques than $G(N,p)$. In order to prove the desired result, we again take $N = \Theta(k^2 n)$ and $p = \Theta(1/k)$ and let $\omega =N^{1/8}$. One can then follow the same proof technique as in \cref{thm:starbursts-ub}, partitioning the vertex set of an edge-colored $G(N,p,\omega)$ into $V_R, V_B,V_0$ as before. The main difference from the proof of \cref{thm:starbursts-ub} is that one must apply Tur\'an's theorem and the Erd\H os--Szekeres argument inside the intersection of one of these sets with an appropriately chosen clique of order $\omega$.

\section{Concluding remarks}

Although we have made substantial progress on the three questions asked by Erd\H os, Faudree, Rousseau, and Schelp, several interesting open problems remain. First, for complete bipartite graphs $K_{s,t}$, though we have shown that $\rhat(K_{s,t}) = \Theta(s^2 t 2^s)$ for $t = \Omega(s\log s)$, we suspect that a similar bound  may hold whenever $s \leq t$. That is, we have the following conjecture.

\begin{conj}\label{conj:kst}
	For all $s \leq t$, $\rhat(K_{s,t}) = \Theta(s^2 t 2^s)$. In particular, $\rhat(K_{t,t}) = \Theta(t^3 2^t)$.
\end{conj}

Secondly, for book graphs $B_n\up k$, where we determined $\rhat(B_n\up k)$ up to a constant factor once $n$ is sufficiently large in terms of $k$, 
it would be interesting to understand how large is sufficient. As discussed near the end of \cref{subsec:books-lb}, our lower bound can be made to work once $n \geq Ck \log k$, but the upper bound 
relies on an application of Szemer\'edi's regularity lemma, resulting in tower-type bounds. In our earlier paper~\cite{CFW}, we were able to avoid applying the regularity lemma when looking at the ordinary Ramsey number of books and it is possible that a similar technique could yield the upper bound $\rhat(B_n\up k) = O(k 2^k n^2)$ for $n$ bounded by a constant-height tower function of $k$. However, it would be interesting to determine whether the same bound holds for $n$ bounded by a single-exponential function of $k$, say.

Thirdly, our determination of the asymptotic order of $\rhat(S_n\up k)$ for starburst graphs $S_n\up k$ raises the question of when the general lower bound $\rhat(H) = \Omega(\chi(H)^2 \Delta(H) \ab{V(H)})$ in \cref{prop:general-starburst-lb}, which holds for any connected graph $H$, is asymptotically tight.

\begin{qu}\label{qu:size-goodness}
	For which families of graphs $H$ does one have $\rhat(H) = \Theta(\chi(H)^2 \Delta(H) \ab{V(H)})$?
\end{qu}

The reason this question may be interesting is that the proof of \cref{prop:general-starburst-lb} uses a version of the Tur\'an coloring. Indeed, we find an appropriate partition of the host graph $G$ into $\chi(H)-1$ parts and then color all internal edges red and all edges between parts blue. The number of parts guarantees that there is no blue copy of $H$, while the choice of partition ensures that no copy of $H$ 
is contained in any of the red parts. 
One can therefore view \cref{qu:size-goodness} as a size Ramsey variant of the so-called Ramsey goodness problem (see, e.g., \cite[Section 2.5]{MR3497267}), which asks when the Tur\'an coloring is extremal for ordinary Ramsey numbers. While it is not completely clear that \cref{qu:size-goodness} is the correct size Ramsey analogue of this problem, it may be a good first step.

It is also natural to consider multicolor size Ramsey numbers.
Formally, given a graph $H$ and an integer $q\geq 2$, let $\rhat(H;q)$ denote the minimum number of edges in a graph $G$ with the property that no matter how the edges of $G$ are colored in $q$ colors, there is always a monochromatic copy of $H$. Most of the results and proofs in this paper can be adapted to deal with the multicolor case. For example, one can show that if $q \geq 2$ is fixed and $t=\Omega(s\log s)$, then 
\[
	\rhat(K_{s,t};q) = \Theta(s^2 t q^s),
\]
with the upper bound again following from a simple averaging argument in an unbalanced complete bipartite graph and the lower bound from a hypergeometric random coloring. Note that here and below all implicit constants may depend on the number of colors $q$.

For starburst graphs, one can show that if $q\geq 3$ is fixed and $n$ is sufficiently large in terms of $k$, then
\[
	\rhat(S_n\up k;q) = \Theta\left(k\cdot n^2\cdot r(K_k;q-1)^2\right),
\]
where $r(K_k;q-1)$ is the $(q-1)$-color (ordinary) Ramsey number of $K_k$. Although it may not look like it at first, this is a natural generalization of the two-color case, since the one-color Ramsey number of $K_k$ is simply $k$. Here, the lower bound is proved analogously to \cref{prop:general-starburst-lb}, except that rather than the Tur\'an coloring, one uses a blowup of a coloring on $r(K_k;q-1)-1$ vertices with no monochromatic $K_k$. For the upper bound, one again uses $G \sim G(N,p)$, but with $p = \Theta(1/k)$ and $N = \Theta(k \cdot n\cdot r(K_k;q-1))$. We partition the vertices of $G$ according to which of the $q$ color classes they have degree at least $3n$ in and then pass to the largest of these $2^q$ sets. If every vertex in this set has at least $3n$ neighbors in all colors, we can again use the Erd\H os--Szekeres argument to build a monochromatic $S_n\up k$. If not, then there is at least one sparse color in this large set. One applies Szemer\'edi's regularity lemma inside this set, followed by Tur\'an's theorem and Ramsey's theorem on the reduced graph to find a non-sparse color and $k$ parts which are pairwise dense and regular in this color. Finally, one can greedily pick out one vertex from each part, inductively maintaining the ability to extend this to a monochromatic copy of $S_n\up k$.

However, it is not clear how to extend our results on book graphs to more than two colors. A natural conjecture is the following. 

\begin{conj}\label{conj:multicolor-books}
Fix $q \geq 3$. For every $k\geq 2$ and all sufficiently large $n$,
\[
	\rhat(B_n\up k;q) = \Theta\left(q^k\cdot n^2\cdot r(K_k;q-1)\right).
\]
\end{conj}

We feel that this is a natural conjecture, because, in the two-color case, both the lower and upper bounds are fundamentally determined by a ``Tur\'an part'' of order $\Theta(kn)$ and a ``random part'' of order $\Theta(2^k n)$ and the natural multicolor analogues of these would have orders $\Theta(r(K_k;q-1)n)$ and $\Theta(q^k n)$, respectively. However, we are not confident of \cref{conj:multicolor-books} even in the case $q=3$ and can prove neither the upper nor the lower bound.

\paragraph{Acknowledgments} We are grateful to Mehtaab Sawhney for suggesting the simplified proof of \cref{lem:lagrange} included here. We would also like to thank the anonymous referees for several helpful comments that improved the presentation of the paper. 

% \bibliographystyle{yuval}
% \bibliography{../size-ramsey}

\begin{thebibliography}{10}
\providecommand{\url}[1]{\texttt{#1}}
\providecommand{\urlprefix}{URL }
\providecommand{\eprint}[2][]{\url{#2}}

\bibitem{MR3327533}
J.~Balogh, R.~Morris, and W.~Samotij, Independent sets in hypergraphs, \emph{J.
  Amer. Math. Soc.} \textbf{28} (2015), 669--709.

\bibitem{MR693028}
J.~Beck, On size {R}amsey number of paths, trees, and circuits. {I}, \emph{J.
  Graph Theory} \textbf{7} (1983), 115--129.

\bibitem{MR1083592}
J.~Beck, On size {R}amsey number of paths, trees and circuits. {II}, in
  {Mathematics of {R}amsey theory}, {Algorithms Combin.}, vol.~5,
  Springer, Berlin, 1990,  34--45.

\bibitem{berger2019size}
S.~Berger, Y.~Kohayakawa, G.~S. Maesaka, T.~Martins, W.~Mendon{\c{c}}a, G.~O.
  Mota, and O.~Parczyk, The size-{R}amsey number of powers of bounded degree
  trees, \emph{J. London Math. Soc.} \textbf{103} (2021), 1314--1332.

\bibitem{clemens2019size}
D.~Clemens, M.~Jenssen, Y.~Kohayakawa, N.~Morrison, G.~O. Mota, D.~Reding, and
  B.~Roberts, The size-{R}amsey number of powers of paths, \emph{J. Graph
  Theory} \textbf{91} (2019), 290--299.

\bibitem{clemens2021grid}
D.~Clemens, M.~Miralaei, D.~Reding, M.~Schacht, and A.~Taraz, On the
  size-{R}amsey number of grid graphs, \emph{Combin. Probab. Comput.}
  \textbf{30} (2021), 670--685.

\bibitem{MR4115773}
D.~Conlon, The {R}amsey number of books, \emph{Adv. Comb.}  (2019), Paper No.
  3, 12pp.

\bibitem{MR3497267}
D.~Conlon, J.~Fox, and B.~Sudakov, Recent developments in graph {R}amsey
  theory, in \emph{Surveys in combinatorics 2015}, \emph{London Math. Soc.
  Lecture Note Ser.}, vol. 424, Cambridge Univ. Press, Cambridge, 2015,
  49--118.

\bibitem{CFW}
D.~Conlon, J.~Fox, and Y.~Wigderson, Ramsey numbers of books and
  quasirandomness, \emph{Combinatorica} \textbf{42} (2022), 309--363.

\bibitem{CFW21}
D.~Conlon, J.~Fox, and Y.~Wigderson, Off-diagonal book {R}amsey numbers,
  to appear in \emph{Combin. Probab. Comput.} Preprint available at
  arXiv:2110.14483 [math.CO].

\bibitem{MR3548529}
D.~Conlon and W.~T. Gowers, Combinatorial theorems in sparse random sets,
  \emph{Ann. of Math.} \textbf{184} (2016), 367--454.

\bibitem{CNT21}
D.~Conlon, R.~Nenadov, and M.~Truji\'{c}, The size-{R}amsey number of cubic
  graphs, \emph{Bull. London Math. Soc.} \textbf{54} (2022), 2135--2150.
  
  \bibitem{CNT23}
D.~Conlon, R.~Nenadov, and M.~Truji\'{c}, On the size-Ramsey number of grids, to appear in \emph{Combin. Probab. Comput.} Preprint 
available at arXiv:2202.01654 [math.CO].

\bibitem{draganic_kolutovi}
N.~Dragani\'{c}, M.~Krivelevich, and R.~Nenadov, Rolling backwards can move you
  forward: on embedding problems in sparse expanders, in {Proceedings of
  the 2021 {ACM}-{SIAM} {S}ymposium on {D}iscrete {A}lgorithms ({SODA})},
  Society for Industrial and Applied Mathematics (SIAM), Philadelphia, PA,
  2021,  123--134.

\bibitem{draganic2020size}
N.~Dragani{\'c}, M.~Krivelevich, and R.~Nenadov, The size-{R}amsey number of
  short subdivisions, \emph{Random Structures Algorithms} \textbf{59} (2021),
  68--78.

\bibitem{MR479691}
P.~Erd\H{o}s, R.~J. Faudree, C.~C. Rousseau, and R.~H. Schelp, The size
  {R}amsey number, \emph{Period. Math. Hungar.} \textbf{9} (1978), 145--161.

\bibitem{MR1212883}
P.~Erd\H{o}s and C.~C. Rousseau, The size {R}amsey number of a complete
  bipartite graph, \emph{Discrete Math.} \textbf{113} (1993), 259--262.

\bibitem{MR905153}
J.~Friedman and N.~Pippenger, Expanding graphs contain all small trees,
  \emph{Combinatorica} \textbf{7} (1987), 71--76.

\bibitem{han2020multicolour}
J.~Han, M.~Jenssen, Y.~Kohayakawa, G.~O. Mota, and B.~Roberts, The multicolour
  size-{R}amsey number of powers of paths, \emph{J. Combin. Theory Ser. B}
  \textbf{145} (2020), 359--375.

\bibitem{han2019size}
J.~Han, Y.~Kohayakawa, S.~Letzter, G.~O. Mota, and O.~Parczyk, The
  size-{R}amsey number of 3-uniform tight paths, \emph{Adv. Combin.}  (2021),
  1--12.

\bibitem{haxell1995induced}
P.~E. Haxell, Y.~Kohayakawa, and T.~{\L}uczak, The induced size-{R}amsey number
  of cycles, \emph{Combin. Probab. Comput.} \textbf{4} (1995), 217--239.

\bibitem{MR1782847}
S.~Janson, T.~{\L}uczak, and A.~Rucinski, \emph{Random graphs},
  Wiley-Interscience Series in Discrete Mathematics and Optimization,
  Wiley-Interscience, New York, 2000.

\bibitem{kamcev2021size}
N.~Kam\v{c}ev, A.~Liebenau, D.~R. Wood, and L.~Yepremyan, The size {R}amsey
  number of graphs with bounded treewidth, \emph{SIAM J. Discrete Math.}
  \textbf{35} (2021), 281--293.

\bibitem{MR2775894}
Y.~Kohayakawa, V.~R\"{o}dl, M.~Schacht, and E.~Szemer\'{e}di, Sparse partition
  universal graphs for graphs of bounded degree, \emph{Adv. Math.} \textbf{226}
  (2011), 5041--5065.

\bibitem{MR2223394}
M.~Krivelevich and B.~Sudakov, Pseudo-random graphs, in {More sets, graphs
  and numbers}, {Bolyai Soc. Math. Stud.}, vol.~15, Springer, Berlin,
  2006,  199--262.

\bibitem{letzter2021size}
S.~Letzter, A.~Pokrovskiy, and L.~Yepremyan, Size-{R}amsey numbers of powers of
  hypergraph trees and long subdivisions. Preprint available at
  arXiv:2103.01942 [math.CO].

\bibitem{MR2321240}
L.~Lov\'{a}sz, \emph{Combinatorial problems and exercises}, second edition, AMS
  Chelsea Publishing, Providence, RI, 2007.

\bibitem{MR1972078}
O.~Pikhurko, Asymptotic size {R}amsey results for bipartite graphs, \emph{SIAM
  J. Discrete Math.} \textbf{16} (2002), 99--113.

\bibitem{MR1767025}
V.~R\"{o}dl and E.~Szemer\'{e}di, On size {R}amsey numbers of graphs with
  bounded degree, \emph{Combinatorica} \textbf{20} (2000), 257--262.

\bibitem{MR3385638}
D.~Saxton and A.~Thomason, Hypergraph containers, \emph{Invent. Math.}
  \textbf{201} (2015), 925--992.

\bibitem{MR3548528}
M.~Schacht, Extremal results for random discrete structures, \emph{Ann. of
  Math.} \textbf{184} (2016), 333--365.

\bibitem{MR1999036}
T.~Szab\'{o} and V.~H. Vu, Tur\'{a}n's theorem in sparse random graphs,
  \emph{Random Structures Algorithms} \textbf{23} (2003), 225--234.

\end{thebibliography}

\end{document}